%% file: arxiv-v3.tex
\documentclass[preprint,12pt]{elsarticle}
\usepackage{amsmath,amssymb,amsthm}
\usepackage[margin=1in]{geometry}
\usepackage{tikz}
\usepackage{float}
\usepackage{comment}
\usepackage{hyperref} 

\newcommand{\R}{\mathbb{R}}
\newcommand{\E}{\mathbb{E}}
\newcommand{\cL}{\mathcal{L}}
\newcommand{\ex}{\operatorname{ex}}
\newcommand{\norm}[1]{\left\lVert #1 \right\rVert}

\definecolor{crossc}{RGB}{0,130,120}   
\definecolor{intc}{RGB}{190,30,45}     
\definecolor{hc}{RGB}{110,110,120}     
\tikzset{
  vtx/.style={circle,fill=black,inner sep=1.6pt},
  crossE/.style={crossc,thick},
  intE/.style={intc,very thick},
  hE/.style={hc,thick},
}

\theoremstyle{definition}
\newtheorem{theorem}{Theorem}
\newtheorem{lemma}[theorem]{Lemma}
\newtheorem{proposition}[theorem]{Proposition}

\newtheorem{problem}[theorem]{Problem}
\newtheorem{corollary}[theorem]{Corollary}
\newtheorem{observation}[theorem]{Observation}
\newtheorem{example}[theorem]{Example}
\newtheorem{remark}[theorem]{Remark}

\theoremstyle{definition}
\numberwithin{theorem}{section}

\begin{document}

  \begin{frontmatter}

    \title{Disconnected graphs and extremal bounds for realizable distance orders}

    \author[label1]{Gerardo L. Maldonado}
    \author[label2]{Leonardo Martínez-Sandoval}
    \author[label3]{Miguel Raggi}
    \author[label4]{Edgardo Roldán-Pensado}

    \affiliation[label1]{organization={Instituto de Matemáticas, UNAM Campus Juriquilla},
      addressline={Blvd. Juriquilla 3001},
      city={Juriquilla},
      postcode={76230},
      state={Querétaro},
      country={Mexico}}
    \affiliation[label2]{organization={Facultad de Ciencias, UNAM},
      addressline={Av. Universidad 3000},
      city={Ciudad de México},
      postcode={04510},
      state={Ciudad de México},
      country={Mexico}}
    \affiliation[label3]{organization={Escuela Nacional de Estudios Superiores, UNAM Campus Morelia},
      addressline={Antigua Carretera a Pátzcuaro 8701, Col. Ex Hacienda San José de la Huerta},
      city={Morelia},
      postcode={58089},
      state={Michoacán},
      country={Mexico}}
    \affiliation[label4]{organization={Centro de Ciencias Matemáticas, UNAM Campus Morelia},
      addressline={Antigua Carretera a Pátzcuaro 8701, Col. Ex Hacienda San José de la Huerta},
      city={Morelia},
      postcode={58089},
      state={Michoacán},
      country={Mexico}}

    \begin{abstract}
      Let $G$ be a graph together with a total order $\prec$ on its edges. We say that $\prec$ is realizable in $\R^d$ if there is a placement of the vertices of $G$ in $\R^d$ such that the Euclidean lengths of the edges induce exactly the order $\prec$. Almendra-Hernández and Martínez-Sandoval proved that every total order on the edges of the complete graph $K_n$ is realizable in $\R^{n-2}$. We show that the same is not true for the disjoint union of two complete graphs: for every $n\geq 3$ there is a total order on the edges of $K_n\sqcup K_n$ that is not realizable in $\R^{n-2}$, but is in $\R^{n-1}$. Thus, the realizability of an order on a disconnected graph is not determined by its restrictions to the connected components. We also study the extremal problem of maximizing the number of edges of an $n$-vertex graph all of whose edge-orders are realizable in $\R^d$, and show that for every fixed $d$ this maximum is $dn+o(n)$, with explicit bounds on the error term. On the real line we sharpen these bounds and show that a disjoint union of two cycles is realizable in $\R$ unless both cycles are triangles.
    \end{abstract}


  \end{frontmatter}

  \section{Introduction}\label{sec:intro}

  \subsection{State of the art}

  Several classical problems in discrete geometry ask what can be said about the distances determined by a finite set of points. For example, there is a long tradition of asking which prescribed arrays of distances admit Euclidean realizations, going back to work of Menger and Schoenberg and continuing in many directions; see, for instance, \cite{SCHOENBERG1935,LIBERTI2014}. Another influential line studies configurations with only a few permitted distances, spherical codes and related structures \cite{DGS1977}. Erdős' distinct distances problem asks how many distinct distances must occur among $n$ points in the plane \cite{GIS2011}, while the theory of unit distance graphs asks which graphs can be drawn in $\R^d$ with all edges of unit length \cite{ALON2014}. Some computational aspects of these problems have also been investigated, for example, in the plane the latter decision problem is complete for the existential theory of the reals \cite{SCHAEFER2013}.

  The question of which orders of the pairwise distances among $n$ points can be realized in a given Euclidean dimension has arisen independently in several settings. Motivated by questions in psychometrics, Shepard~\cite{SHEPARD1962} studied it in the context of nonmetric multidimensional scaling and observed that every prescribed order, including ties, can be realized in $\R^{n-1}$. Bilu and Linial~\cite{BL05} revisited the question in the context of monotone embeddings of finite metric spaces, asking how much the dimension can be reduced while preserving the order of all pairwise distances. Almendra-Hernández and Martínez-Sandoval~\cite{AM22} formulated the problem in terms of prescribed distance orders and showed, by a geometric gluing argument, that $\R^{n-2}$ suffices for every strict order when $n\geq3$. The same bound for strict orders had appeared earlier, in a theorem of Guttman for which Lingoes gave a spectral proof~\cite{GUTTMAN1968,LINGOES1971}. Morić and Pach~\cite{MORICPACH2013} study the distances determined by a point set, listed in decreasing order. In ordinal embedding, distance comparisons are used to recover the geometry of a point configuration~\cite{KLEINDESSNER2014,ARIASCASTRO2017}.

  Here we fix only a set of pairs, represented by the edges of a graph, and ask in what Euclidean dimension every total order on those edges can be realized by their lengths. Let $G=(V,E)$ be a graph. Consider an injection $p\colon V\to\R^d$, such that the edge lengths $\norm{p(u)-p(v)}$ with $uv\in E$ are pairwise distinct. Such a placement induces a total order on $E$ by comparing lengths:
  \begin{equation*}
    uv \prec xy \iff \norm{p(u)-p(v)} < \norm{p(x)-p(y)}.
  \end{equation*}
  Given a total order $\prec$ on $E$, we say that $\prec$ is \emph{realizable in $\R^d$} if it is induced by some such placement.

  Two observations are immediate. First, a realization in $\R^d$ is also a realization in $\R^{d+1}$, so realizability can only improve as the dimension grows. Second, if $(G,\prec)$ is realizable in $\R^d$ and $H$ is a subgraph of $G$, then $H$ with the induced order is realizable there as well: the same placement works. Thus, it is natural to associate to every graph $G$ the parameter
  \begin{equation*}
    R(G) = \min\{d : \text{every total order on $E(G)$ is realizable in $\R^d$}\},
  \end{equation*}
  which is always finite, as we discuss below. We say that the graph $G$ itself is \emph{realizable in $\R^d$} if every total order on $E(G)$ is, so that $R(G)$ is the least dimension in which $G$ is realizable.

  For some basic graph families the parameter is easy to determine. As we will see in Section~\ref{sec:basic}, trees, and more generally forests, as well as cycles are realizable already on the real line. The problem of determining $R(K_n)$ is substantially harder. The parameter $d(n,\ell_2)$ of Bilu and Linial~\cite{BL05} is exactly $R(K_n)$. They proved $\lfloor n/2\rfloor\leq R(K_n)\leq n$, showed by a counting argument that, for every fixed $\delta>0$, the proportion of orders on $E(K_n)$ realizable in dimension less than $n/(2+\delta)$ tends to zero as $n\to\infty$, and asked whether the gap can be closed. Almendra-Hernández and Martínez-Sandoval~\cite{AM22} proved the upper bound $R(K_n)\leq n-2$, for $n\geq 3$. Hence, for any graph $G$ on $n\geq 3$ vertices it follows that $R(G)\leq n-2$ as well. Whether $R(K_n)=n-2$ for every $n\geq3$ remains open (it holds for $3\leq n\leq6$, see below). For every $n\geq7$, Maldonado, Raggi and Roldán-Pensado~\cite[Proposition~6]{MRR26} realized in $\R^{n-3}$ a total order satisfying the conditions proposed in~\cite{AM22} for a lower bound.

  The bipartite case is better understood: for $2\leq n\leq m$, in \cite{AM22} it is shown that $n-1\leq R(K_{n,m})\leq n$, and in \cite{MRR26} that $R(K_{n,m})=n$ whenever $n<m$. The balanced case $n=m$ also remains open. Since $K_{a,n-a}$ is a subgraph of $K_n$, this yields $R(K_n)\geq\lfloor(n-1)/2\rfloor$. A better bound comes from \emph{sphericity}: the sphericity $\mathrm{Sph}(H)$ of a graph $H$ is the least $d$ such that $H$ is the unit-distance neighbourhood graph of some points in $\R^d$, with $uv\in E(H)$ if and only if $\norm{p(u)-p(v)}\leq1$. Split the vertices of $K_n$ into two parts and take any order in which every edge between the parts precedes every edge inside a part. After rescaling, a realization of such an order is a unit-distance representation of the complete bipartite graph between the parts. This is how Bilu and Linial obtain $R(K_n)\geq\lfloor n/2\rfloor$ from $\mathrm{Sph}(K_{m,m})\geq m$ \cite{RRS89}. Maehara \cite[Theorem~2 and the example after Lemma~2]{MAEHARA1991} had in fact shown $\mathrm{Sph}(K_{m,m})>m$ for $m\geq3$ and $\mathrm{Sph}(K_{2,3})>2$, which gives
  \begin{equation*}
    R(K_n)\geq\lfloor n/2\rfloor+1\qquad\text{for } n\geq5,
  \end{equation*}
  the best general lower bound we are aware of. Together with \cite{AM22}, it gives $R(K_5)=3$ and $R(K_6)=4$. We remark that the worst-case dimension becomes $n-1$ if ties among distances must be preserved, as for ultrametrics \cite{ABDFHS08}; this does not contradict $R(K_n)\leq n-2$, since here all edge lengths are distinct. Beyond these families, essentially nothing is known, and computing $R(G)$ for a general graph appears to be a difficult problem.

  \subsection{Our contributions}
  We start a systematic study of the parameter $R(G)$. In Section~\ref{sec:basic} we state basic results concerning bridges and cut vertices, which show that realizability can be reduced to studying the blocks of a graph. We also show that if the $\prec$-smallest edge $e$ of $G$ does not lie in a triangle, then realizability of the induced order on the contraction $G/e$ implies realizability of $\prec$ on $G$.

  Then, we turn our attention to disconnected graphs. One would expect that the realizability of a graph depends on the realizability of its connected components. This is partially true. When each connected component of $G$ has order at most $n$, we show that $R(G)\leq n-1$. However, an unintuitive new phenomenon appears. It is not true that realizability of each connected component of $G$ in dimension $d$ implies $R(G)\leq d$. This happens because a total order compares edges from different components and then the sets of edge lengths must also interleave correctly.

  The simplest instance is that $K_3$ is realizable on the real line, while $K_3 \sqcup K_3$ is not. More generally, since $R(K_n)\leq n-2$, it is natural to ask whether every total order on the edges of the disjoint union $K_n\sqcup K_n$ is also realizable in $\R^{n-2}$. Our first main result says that it is not, that one extra dimension is necessary and sufficient.

  \begin{theorem}\label{thm:KnUKn}
    For every $n\geq3$,
    \begin{equation*}
      R(K_n\sqcup K_n)=n-1.
    \end{equation*}
  \end{theorem}

  The lower bound in Theorem~\ref{thm:KnUKn} follows by combining sharp geometric inequalities of Maehara and Deza--Maehara~\cite{MAEHARA1991,DEZAMAEHARA1994} with a classical theorem of Schütte~\cite{SCHUTTE1963}. In Section~\ref{sec:disconnected} we also discuss these inequalities and provide an alternative probabilistic proof for both of them from a common variance identity. The upper bound follows from the positive connected component result mentioned above.

  The next part of our work is motivated by extremal questions on the maximal number of edges that a distance-related graph may have. Two canonical examples are the largest number of unit-distance pairs \cite{erdos1946sets} and the largest number of diameters of a point set \cite{SWANEPOEL2009}. We ask how many edges a graph of order $n$ may have if it is realizable in $\R^d$. Therefore, we define
  \begin{equation*}
    \ex_{\R^d}(n)
    =
    \max\bigl\{|E(G)|\ :\ G\text{ is a graph of order $n$ and $R(G)\leq d$}\bigr\}.
  \end{equation*}
  We find a clean answer up to first order.

  \begin{theorem}\label{thm:exd}
    For every fixed $d\geq1$,
    \begin{equation*}
      \ex_{\R^d}(n)=dn+o(n).
    \end{equation*}
  \end{theorem}

  In Section~\ref{sec:higher} we give explicit bounds on the \emph{error term} $E_d(n):=\ex_{\R^d}(n)-dn$. The lower bound follows from a \textit{snake} construction in Section~\ref{sec:basic}. The upper bound adapts the sign-pattern argument of Bilu and Linial~\cite{BL05} to arbitrary edge sets and gives the leading coefficient $d$ in the extremal bound.

  On the real line both sides of the problem become sharper. We begin by looking more closely at disjoint unions of cycles. For two cycles, the only obstruction is $C_3\sqcup C_3$.

  \begin{theorem}\label{thm:twocycles} Let $m,n\geq3$. The disjoint union $C_m\sqcup C_n$ is realizable in $\R$ if and only if $(m,n)\neq(3,3)$. \end{theorem}

  This result suggests that if we have a disjoint union of a fixed number of cycles, once the cycles are sufficiently long the graph should be realizable on the real line. We prove that this is the case, more precisely:

  \begin{theorem}\label{thm:manycycles} For every $k\geq1$ there exists $K$ such that the disjoint union $kC_K$ is realizable in $\R$. In fact, one may take \[ K=2\cdot3^{k-1}+2. \] \end{theorem}

  Theorem~\ref{thm:manycycles} allows us to construct an example with substantially more edges than the one supplied by the snake construction.

  The upper bound can also be sharpened on the real line. Here, in the semialgebraic region-counting argument, differences of squared distances factor into linear forms, so we can use a hyperplane-arrangement argument. Combining these two improvements gives the following bounds.

  \begin{theorem}\label{thm:ex1} We have \begin{equation*} n+\frac{\ln n}{\ln3}-O(\ln\ln n) \leq \ex_{\R}(n) \leq n+\frac{2n}{\ln n} +O\!\left(\frac{n}{(\ln n)^2}\right). \end{equation*} \end{theorem}

  \subsection{Organization of the paper}
  The paper is organized as follows. Section~\ref{sec:basic} collects basic reductions and the simplex-gluing procedure used throughout the paper. Section~\ref{sec:disconnected} develops the first surprise caused by disconnected components and proves Theorem~\ref{thm:KnUKn}. Section~\ref{sec:higher} treats the extremal problem in arbitrary dimension and proves Theorem~\ref{thm:exd}, including quantitative bounds on the error term. Section~\ref{sec:line} turns to the real line, where we prove Theorems~\ref{thm:twocycles}, ~\ref{thm:manycycles} and~\ref{thm:ex1}. Finally, Section~\ref{sec:final} states some final remarks and collects open problems.

  \section{Basic properties of realizability}\label{sec:basic}

  We begin by collecting elementary reductions and a gluing procedure that will be used several times later. The reductions give a feeling for the geometry of the problem and also explain why disconnected graphs arise naturally even when one starts from a connected graph.

  \begin{observation} \label{obs:easy}
    The following elementary operations relate the realizability of a graph to that of a smaller one. We record them here since they give a feeling for the geometry of the problem and single out the graphs for which it is actually hard. In all of them, the deleted or contracted edge is removed from the order, and the order on the surviving edges is kept.
    \begin{itemize}
      \item \emph{Bridges can be erased.} If $e$ is a bridge of $G$, then $(G,\prec)$ is realizable in $\R^d$ if and only if $(G-e,\prec)$ is. Indeed, after deleting $e$ the two sides of the bridge share no edges, so one of them may be translated freely without changing any other edge length. In this way $e$ can be put back with whatever length the order demands.
      \item \emph{Cut vertices can be split.} Suppose $G$ is the union of two subgraphs $G_1$ and $G_2$ whose only common vertex is $v$. Then $(G,\prec)$ is realizable in $\R^d$ if and only if the disjoint union $G_1\sqcup G_2$, with $v$ duplicated and with the same order on the edges, is. Again translations do the job: gluing the two copies of $v$ back together changes no edge length, and an arbitrarily small perturbation keeps the placement injective. Consequently, realizability depends only on the disjoint union of the blocks (the maximal $2$-connected subgraphs) of $G$.
      \item \emph{The smallest edge can be contracted.} Let $e=uv$ be the $\prec$-smallest edge of $G$, and suppose it is not in a triangle in $G$. Consider $G/e$ with the inherited order. If $(G/e,\prec)$ is realizable in $\R^d$, then so is $(G,\prec)$. In a realization of the contraction, split the contracted vertex into two points at distance $\varepsilon$. If $\varepsilon$ is small enough, every strict comparison survives, and $e$ itself is the shortest edge, as required. The converse fails in general. We will usually only contract edges not in triangles without explicitly verifying this hypothesis.
    \end{itemize}
  \end{observation}

  In particular, leaf edges may be removed as bridges. These operations already settle realizability for all forests and cycles. In a forest every edge is a bridge, so every total order on its edges is realizable even on the real line: that is, $R(F)=1$ for every forest $F$ with at least one edge.

  For a cycle $C_k$ with $k\geq 4$, contracting the $\prec$-smallest edge produces $C_{k-1}$, and every order on the three edges of a triangle is easily realizable in $\R$ (the $\prec$-largest edge must simply join the two extreme points). By induction, $R(C_k)=1$ for every $k\geq 3$.

  We will also use a simple gluing procedure. Recall that a \textit{$d$-tree} is a graph obtained from $K_{d+1}$ by repeatedly adding a new vertex adjacent to all the vertices of some $d$-clique already present. Equivalently, it is obtained by gluing copies of $K_{d+1}$ along copies of $K_d$ in a tree-like fashion.

  \begin{proposition}\label{prop:dtree-realizable}
    Every linear order on the edges of a finite $d$-tree is realizable by
    distances of an injective configuration in $\R^d$.
  \end{proposition}

  \begin{proof}
    Let
    \[
    e_1\prec e_2\prec\cdots\prec e_m
    \]
    be the prescribed order, and assign to $e_i$ the target length
    \[
    \ell(e_i)=1+i\varepsilon.
    \]
    We take $\varepsilon>0$ sufficiently small.  Indeed, the edge lengths
    of a regular $d$-simplex are all equal to $1$, and realizability by a
    nondegenerate $d$-simplex is an open condition on its edge lengths \cite{dekster1987edge}.
    Thus, for sufficiently small $\varepsilon$, the restriction of
    $\ell$ to every copy of $K_{d+1}$ occurring in the construction of
    the $d$-tree is realized by a simplex in $\R^d$.

    Realize one of these simplices first, and proceed along a construction
    tree.  Suppose that a new copy of $K_{d+1}$ is attached along a
    $d$-clique $C$ that has already been realized.  Realize the new
    $K_{d+1}$ independently with its prescribed edge lengths.  Its copy
    of $C$ has exactly the same pairwise distances as the copy already
    present, so the two are congruent.  Hence an isometry of $\R^d$
    identifies them.  Applying this isometry to the whole new simplex
    glues it to the existing configuration without changing any edge length that we have so far defined.

    Iterating this procedure yields a map $ p_0:V(G)\longrightarrow\R^d$
    satisfying
    \[
    \|p_0(x)-p_0(y)\|=1+i\varepsilon
    \qquad\text{whenever }xy=e_i.
    \]
    Vertices lying in different simplices may accidentally have the same
    image, so $p_0$ need not be injective.  Since the prescribed edge
    lengths are separated by gaps of $\varepsilon$, however, a
    sufficiently small generic perturbation makes all vertex images
    distinct while preserving every strict comparison among edge lengths.
    The resulting configuration realizes the prescribed order.
  \end{proof}

  \begin{remark}\label{rem:partial-dtree-realizable}
    The requirement that consecutive simplices be glued along facets is
    not essential. The same conclusion holds if copies of $K_{d+1}$ are
    glued in a tree-like fashion along arbitrary complete subgraphs
    $K_r$, with $r\leq d$. The proof above applies without change: for the gluing phase we just need an isometry that identifies fewer vertices.
  \end{remark}

  We record two examples that will be used later.

  \begin{example}[Snakes of simplices]\label{ex:snake}
    Take a $d$-tree whose simplex pieces form a path, that is, a \emph{snake} of $d$-simplices (Figure~\ref{fig:snake}). Starting with $K_{d+1}$, every additional vertex contributes exactly $d$ new edges. Hence a snake on $n\geq d+1$ vertices has
    \[
    \binom{d+1}{2}+d(n-d-1)
    =dn-\binom{d+1}{2}
    \]
    edges, and Proposition~\ref{prop:dtree-realizable} shows that every linear order on them is realizable in $\R^d$.
  \end{example}

  \begin{figure}[H]
    \centering
    \input{figuras/snake}
    \caption{The construction of Example~\ref{ex:snake}: a \emph{snake} of $d$-simplices, drawn for $d=2$ and $d=3$. One starts from a single copy of $K_{d+1}$, shaded darkest on the left, and repeatedly attaches a new vertex $v$ joined to all $d$ vertices of a $d$-clique of the most recent simplex; the \textcolor{intc}{$d$ edges} contributed by the last such step are highlighted. A snake on $n$ vertices has $dn-\binom{d+1}{2}$ edges, and by Proposition~\ref{prop:dtree-realizable} every total order on them is realizable in $\R^d$.}
    \label{fig:snake}
  \end{figure}
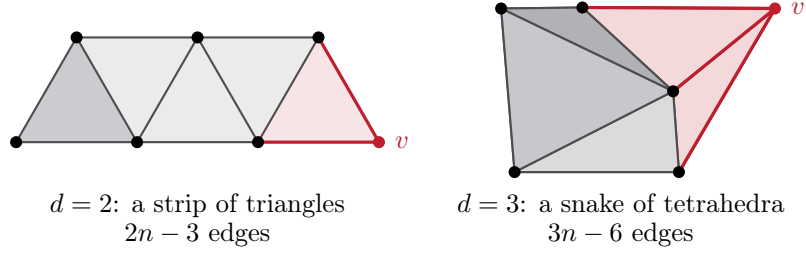

  \begin{example}[Disjoint unions of simplices]\label{ex:disjoint-simplices}
    Every linear order on the edges of a disjoint union of copies of $K_{d+1}$ is realizable in $\R^d$. Indeed, use Remark~\ref{rem:partial-dtree-realizable} with $r=0$, i.e., with empty identifications.
  \end{example}

  As a corollary, we can upper bound $R(G)$ for $G$ with small connected components.

  \begin{corollary}[Small connected components]\label{cor:small-cc}
    If every connected component of a graph $G$ has order at most $d+1$, then $R(G)\leq d$.
  \end{corollary}

  This shows some disconnected graphs are extremely well behaved. We now turn to the first place where this breaks down.

  \section{The value of $R(K_n\sqcup K_n)$}\label{sec:disconnected}

  The first surprise already appears for two triangles.

  \begin{proposition}\label{prop:two-triangles}
    If a graph $G$ contains two distinct triangles, then $G$ is not realizable in $\R$.
  \end{proposition}

  \begin{proof}
    Observe that in any placement of three points on the line, the longest of the three pairwise distances equals the sum of the other two.

    Two triangles in a graph are either disjoint, share exactly one vertex, or share an edge. By Observation~\ref{obs:easy} the second case reduces to the first, so it is enough to exhibit an unrealizable order in each of the other two.

    Figure~\ref{fig:two-triangles} shows the orders. For simplicity, we label edges with numbers to show their relative total order.

    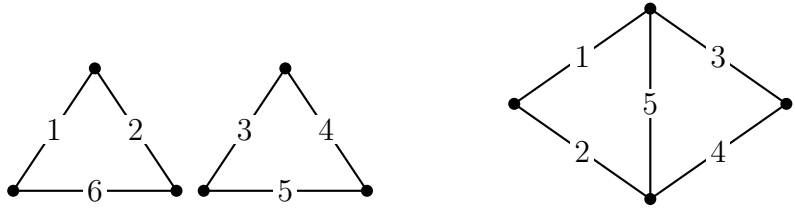
\begin{figure}[H]
      \centering
      \input{figuras/triangles-disjoint}\qquad\qquad\input{figuras/triangles-shared}
      \caption{Two configurations witnessing Proposition~\ref{prop:two-triangles}. Edges are labelled by their rank in the prescribed order. \emph{Left:} two disjoint triangles. \emph{Right:} two triangles sharing the edge $5$.}
      \label{fig:two-triangles}
    \end{figure}

    Write $\ell_i$ for the length of the edge labelled $i$. In each triangle the longest of its three edges must be the sum of the other two. For the disjoint triangles this gives $\ell_6=\ell_1+\ell_2$ and $\ell_5=\ell_3+\ell_4$, so
    \begin{equation*}
      \ell_6=\ell_1+\ell_2<\ell_3+\ell_4=\ell_5<\ell_6,
    \end{equation*}
    a contradiction. For the two triangles sharing the edge $5$, which is the longest edge of both, it gives $\ell_1+\ell_2=\ell_5=\ell_3+\ell_4$, contradicting $\ell_1<\ell_3$ and $\ell_2<\ell_4$.
  \end{proof}

  In particular, $R(C_3)=1$ while $R(C_3\sqcup C_3)>1$: realizability of the two components separately does not suffice to realize an arbitrary interleaving of their edge lengths. Theorem~\ref{thm:KnUKn}, which we now prove, says that the same phenomenon persists in every dimension for two complete graphs.

  The upper bound is already contained in Example~\ref{ex:disjoint-simplices}. The orders witnessing the lower bound are constructed as follows. Denote the two copies of $K_n$ by $G$ and $H$, and split the vertex set of $G$ into two parts $A$ and $B$ with $|A|=\lfloor n/2\rfloor$ and $|B|=\lceil n/2\rceil$. We call an edge of $G$ \emph{internal} if both of its endpoints lie in the same part, and \emph{cross} otherwise. Let $\prec$ be any total order on the edges of $K_n\sqcup K_n$ such that
  \begin{equation}\label{eq:blocks}
    e \prec f \prec e'
  \end{equation}
  whenever $e$ is a cross edge of $G$, $f$ is an edge of $H$, and $e'$ is an internal edge of $G$. In other words, the edges come in three blocks: first all the cross edges of $G$, then all the edges of $H$, and finally all the internal edges of $G$. Within each block the order may be refined arbitrarily (see Figure~\ref{fig:order}).

  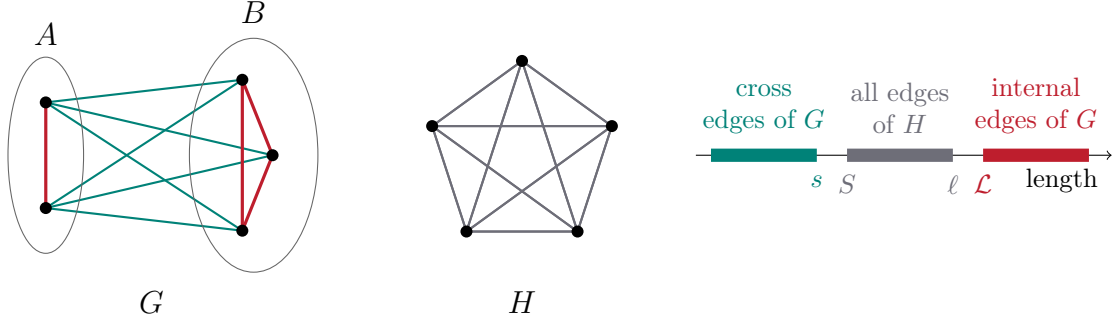
\begin{figure}[H]
    \centering
    \input{figuras/order}
    \caption{The order for $n=5$. The six \textcolor{crossc}{cross edges} of $G$ form the shortest block, the ten \textcolor{hc}{edges of $H$} come next, and the four \textcolor{intc}{internal edges} of $G$ form the longest block. Within each block the order is arbitrary. The four lengths $s$, $S$, $\ell$ and $\cL$ used in the proof of Theorem~\ref{thm:KnUKn} are marked below the axis.}
    \label{fig:order}
  \end{figure}

  We claim that no order satisfying \eqref{eq:blocks} is realizable in $\R^{n-2}$. The idea of the proof is simple. On one hand, in a point set split into two parts as $G$ is, the distances within a part \emph{cannot be too far} from the distances between the parts. On the other hand, any $n$ points in $\R^{n-2}$ have their shortest and longest distances \emph{somewhat far apart}. The order \eqref{eq:blocks} makes these two facts incompatible.

  To make this idea precise, suppose that some placement of the vertices of $K_n\sqcup K_n$ in $\R^{n-2}$ realizes such an order, and consider the following four edge lengths, marked in Figure~\ref{fig:order}: the longest cross edge $s$ of $G$, the shortest and longest edges $S$ and $\ell$ of $H$, and the shortest internal edge $\cL$ of $G$ ($s$ for small and $\ell$ for large). The order forces
  \begin{equation}\label{eq:sandwich}
    s < S < \ell < \cL.
  \end{equation}
  We now need to state two geometric inequalities depending on the following parameters. For integers $a,b\geq 1$ with $a+b\geq 3$, define
  \begin{equation*}
    \theta_{a,b} = \frac{\binom{a}{2}}{a^2}+\frac{\binom{b}{2}}{b^2} = 1-\frac{1}{2a}-\frac{1}{2b},
  \end{equation*}
  and for $n\geq 3$ set
  \begin{equation*}
    \theta_n = \theta_{\lfloor n/2\rfloor,\lceil n/2\rceil}.
  \end{equation*}
  We have $0<\theta_{a,b}<1$, and among all $a,b$ with $a+b=n$ the quantity $\theta_{a,b}$ is largest when the two parts are as balanced as possible. Thus
  \begin{equation}\label{eq:theta-balance}
    \theta_{a,b}\leq \theta_n<1 \qquad\text{whenever }a+b=n.
  \end{equation}

  The first inequality says that two clusters cannot be pulled apart from each other too much. In the balanced case it is the complete-bipartite obstruction of Maehara~\cite[Theorem~5]{MAEHARA1991}, reproved by Deza and Maehara~\cite[Theorem~6]{DEZAMAEHARA1994}; we discuss the provenance of the general form after stating both lemmas. On the copy $G$ alone, the order~\eqref{eq:blocks} is the sphericity order described in the introduction, which only forces dimension about $n/2$; the second copy $H$ is what pushes the bound to $n-1$.

  \begin{lemma}[Two-cluster inequality; cf.~\cite{MAEHARA1985,MAEHARA1991,DEZAMAEHARA1994}]\label{lem:gap}
    Let $A$ and $B$ be disjoint finite sets of points in $\R^d$ with $|A|=a$ and $|B|=b$, where $a+b\geq 3$. Let $s$ be the largest distance between a point of $A$ and a point of $B$, and let $\cL$ be the smallest distance between two distinct points lying in the same set. Then
    \begin{equation*}
      \theta_{a,b}\,\cL^2\leq s^2.
    \end{equation*}
  \end{lemma}

  The second inequality is where the dimension comes in: in $\R^{n-2}$, a set of $n$ points cannot be nearly equidistant.

  \begin{lemma}[Schütte \cite{SCHUTTE1963}]\label{lem:spread}
    Let $n\geq 3$ and let $P$ be a set of $n$ points in $\R^{n-2}$. If $S$ is the smallest and $\ell$ is the largest distance between two distinct points of $P$, then
    \begin{equation*}
      S^2\leq \theta_n\,\ell^2.
    \end{equation*}
  \end{lemma}

  Both inequalities are classical, although their usual formulations and proofs look rather different. Writing $c(X)$ for the circumradius of a finite set $X$, Maehara~\cite[Lemma~1]{MAEHARA1991} shows that $c(A)^2+c(B)^2\leq s^2$, while~\cite[Corollary~2]{MAEHARA1985}, applied to $A$ and to $B$ after rescaling, gives $c(A)\geq\cL\sqrt{(a-1)/(2a)}$ and $c(B)\geq\cL\sqrt{(b-1)/(2b)}$; since $\frac{a-1}{2a}+\frac{b-1}{2b}=\theta_{a,b}$, the two combine to give Lemma~\ref{lem:gap}. Deza and Maehara~\cite[Theorem~6]{DEZAMAEHARA1994} obtain the balanced case directly from the negative-type inequality for squared distances, by the same averaging that we use below. The second inequality is Schütte's theorem~\cite{SCHUTTE1963}, for which further proofs were given by Schoenberg~\cite{SCHOENBERG1969}, Seidel~\cite{SEIDEL1969}, and Bárány~\cite{BARANY1994}. We nevertheless include short proofs of both. Besides keeping the argument self-contained, we find it useful that the two estimates needed here can be obtained economically from the same standard variance identity. The proof of Lemma~\ref{lem:spread} closely follows Bárány's argument~\cite{BARANY1994}, as presented by Swanepoel~\cite{SWANEPOEL2014}.

  We use probabilistic notation only as shorthand for finite weighted averages. If $P=\{P_1,\dots,P_k\}$ is a finite set of points, we also write $P$ for the random point uniformly distributed on this set. More generally, for a probability vector $\alpha=(\alpha_1,\dots,\alpha_k)$ we write $P_\alpha$ for the random point that takes the value $P_i$ with probability $\alpha_i$. Thus $\E[P_\alpha]$ is a point of $\operatorname{conv}(P)$, and every point of $\operatorname{conv}(P)$ arises in this way.

  If $X$ and $Y$ are independent random points, and $X'$ and $Y'$ are independent copies of $X$ and $Y$, respectively, then a standard variance identity gives
  \begin{equation}\label{eq:variance-identity}
    \E\norm{X-Y}^2
    =
    \frac12\E\norm{X-X'}^2
    +\frac12\E\norm{Y-Y'}^2
    +\norm{\E X-\E Y}^2.
  \end{equation}
  In the uniform case this is, after clearing denominators, the identity of Deza and Maehara~\cite[Theorem~1]{DEZAMAEHARA1994}, who use it in exactly this way. Both lemmas above follow from it in a few lines.

  We will also use the elementary estimate
  \begin{equation}\label{eq:collision}
    \Pr[X\neq X']\leq 1-\frac1k
  \end{equation}
  whenever $X$ takes at most $k$ values. Indeed, if these values have probabilities $p_1,\dots,p_k$, then $\Pr[X=X']=\sum_i p_i^2$, and the rearrangement inequality gives $\sum_i p_ip_{\sigma(i)}\leq\sum_i p_i^2$ for every permutation $\sigma$. Summing this over the $k$ cyclic shifts yields
  \begin{equation*}
    1=\Bigl(\sum_i p_i\Bigr)^2\leq k\sum_i p_i^2=k\,\Pr[X=X'].
  \end{equation*}

  \begin{proof}[Proof of Lemma~\ref{lem:gap}]
    Regard $A$ and $B$ as independent uniform random points. Let $A'$ and $B'$ be independent copies of them. Then
    \begin{align*}
      s^2
      &\geq \E\norm{A-B}^2 && \text{(the maximum is at least the mean)}\\
      &= \frac12\E\norm{A-A'}^2
      +\frac12\E\norm{B-B'}^2
      +\norm{\E A-\E B}^2 && \text{by \eqref{eq:variance-identity}}\\
      &\geq \frac12\E\norm{A-A'}^2
      +\frac12\E\norm{B-B'}^2\\
      &\geq \frac12\left(1-\frac1a\right)\cL^2
      +\frac12\left(1-\frac1b\right)\cL^2\\
      &= \theta_{a,b}\,\cL^2.
    \end{align*}
    For the penultimate line, $A=A'$ with probability $1/a$, while whenever $A\neq A'$ the two points lie in the same part and hence are at distance at least $\cL$. The same argument applies to $B$.
  \end{proof}

  \begin{proof}[Proof of Lemma~\ref{lem:spread}]
    By Radon's theorem \cite{RADON1921} (see also \cite{MAT2002}) there is a partition $P=A\sqcup B$ into nonempty parts whose convex hulls intersect. Writing $|A|=a$ and $|B|=b$, where $a+b=n$, choose probability vectors $\alpha$ and $\beta$ such that
    \begin{equation*}
      \E[A_\alpha]=\E[B_\beta].
    \end{equation*}
    Hence the last term in \eqref{eq:variance-identity} vanishes, and
    \begin{align*}
      S^2
      &\leq \E\norm{A_\alpha-B_\beta}^2 && \text{(the minimum is at most the mean)}\\
      &= \frac12\E\norm{A_\alpha-A_\alpha'}^2
      +\frac12\E\norm{B_\beta-B_\beta'}^2\\
      &\leq \frac12\Bigl(\Pr[A_\alpha\neq A_\alpha']
      +\Pr[B_\beta\neq B_\beta']\Bigr)\ell^2\\
      &\leq \frac12\left(1-\frac1a\right)\ell^2
      +\frac12\left(1-\frac1b\right)\ell^2 && \text{by \eqref{eq:collision}}\\
      &= \theta_{a,b}\,\ell^2
      \leq \theta_n\,\ell^2,
    \end{align*}
    where the last inequality is \eqref{eq:theta-balance}.
  \end{proof}

  We can now complete the proof of Theorem~\ref{thm:KnUKn}.

  \begin{proof}[Proof of Theorem~\ref{thm:KnUKn}]
    The upper bound follows from Example~\ref{ex:disjoint-simplices}. For the lower bound, let $\prec$ be any total order satisfying \eqref{eq:blocks}, and suppose for the sake of contradiction that some placement of the vertices of $K_n\sqcup K_n$ in $\R^{n-2}$ realizes $\prec$. Identify each vertex with its image, so that the parts $A$ and $B$ of $G$ become disjoint point sets in $\R^{n-2}$ and $H$ becomes a set of $n$ points in $\R^{n-2}$, and let $s$, $S$, $\ell$ and $\cL$ be the four lengths defined above. Applying Lemma~\ref{lem:gap} to $A$ and $B$, applying Lemma~\ref{lem:spread} to the $n$ points of $H$, and combining with \eqref{eq:sandwich}, we obtain
    \begin{equation*}
      \theta_n\, \cL^2 \;\leq\; s^2 \;<\; S^2 \;\leq\; \theta_n\, \ell^2 \;<\; \theta_n\, \cL^2,
    \end{equation*}
    which is a contradiction. No order satisfying \eqref{eq:blocks} is therefore realizable in $\R^{n-2}$, and $R(K_n\sqcup K_n)>n-2$.
  \end{proof}

  \begin{remark}

    There is actually a whole family of bad orders. Writing $a=\lfloor n/2\rfloor$ and $b=\lceil n/2\rceil$, the orders satisfying \eqref{eq:blocks} number
    \begin{equation*}
      (ab)!\;\binom{n}{2}!\;\Bigl(\binom{a}{2}+\binom{b}{2}\Bigr)!,
    \end{equation*}
    out of the $\bigl(2\binom{n}{2}\bigr)!$ total orders on the edges of $K_n\sqcup K_n$.
  \end{remark}

  \begin{remark}\label{rem:sharp}
    Both constants are optimal, and for the same configuration. Let $A$ and $B$ be regular simplices with a common edge length $\cL$, centred at the same point and spanning orthogonal subspaces of dimensions $a-1$ and $b-1$, so that the whole configuration lies in $\R^{n-2}$. Every internal distance then equals $\cL$, while every cross distance equals $s$, where, writing $r_A=\cL\sqrt{(a-1)/(2a)}$ and $r_B=\cL\sqrt{(b-1)/(2b)}$ for the two circumradii,
    \begin{equation*}
      s^2=r_A^2+r_B^2=\Bigl(\tfrac{a-1}{2a}+\tfrac{b-1}{2b}\Bigr)\cL^2=\theta_{a,b}\,\cL^2.
    \end{equation*}
    So the two parts can be pulled apart until $\theta_{a,b}\,\cL^2=s^2$, and Lemma~\ref{lem:gap} is sharp even in $\R^{n-2}$. Schütte proves more than the inequality of Lemma~\ref{lem:spread}: equality holds exactly for these configurations, with $a=\lfloor n/2\rfloor$ and $b=\lceil n/2\rceil$. In particular $\theta_n$ is the optimal constant there as well.
  \end{remark}

  \begin{corollary}\label{cor:tKn}
    For every $n\geq3$ and every $t\geq2$,
    \begin{equation*}
      R(tK_n)=n-1.
    \end{equation*}
  \end{corollary}

  \begin{proof}
    The upper bound is exactly Example~\ref{ex:disjoint-simplices}. The lower-bound argument above already rules out $\R^{n-2}$ as soon as two copies of $K_n$ are present.
  \end{proof}

  \section{The extremal problem in $\R^d$}\label{sec:higher}

  We now turn to Theorem~\ref{thm:exd}. The arguments below provide the following estimates for the error term $E_d(n)$
  \begin{equation}\label{eq:exd-quantitative}
    -\binom{d+1}{2}
    \leq E_d(n)\leq
    (2+2\ln2)\,\frac{dn}{\ln (dn)}+O\!\left(\frac{dn}{(\ln (dn))^{2}}\right),
  \end{equation}
  with an absolute implied constant. For fixed $d$, this proves the $dn+o(n)$ statement of Theorem~\ref{thm:exd}.

  The lower estimate is exactly the construction in Example~\ref{ex:snake}: for every $n\geq d+1$, the snake on $n$ vertices has $dn-\binom{d+1}{2}$ edges and every linear order on them is realizable in $\R^d$.

  For the upper bound we count the possible orders by the sign conditions of the distance-comparison polynomials. A placement is now a point of $\R^{dn}$, and the comparison of two edges $uv$ and $xy$ is given by the sign of
  \begin{equation*}
    Q_{uv,xy}(p)=\norm{p(u)-p(v)}^{2}-\norm{p(x)-p(y)}^{2},
  \end{equation*}
  a quadratic polynomial. The strict sign pattern of the polynomials $Q_{uv,xy}$ determines the edge order. This counting argument goes back to Bilu and Linial~\cite{BL05}, who used it, with a variant of the Milnor--Thom bound due to Alon, Frankl and Rödl, to show that almost every order on $E(K_n)$ needs dimension at least $n/(2+\delta)$. Here we apply it to arbitrary graphs and keep track of the leading constant. By Warren's theorem~\cite{WARREN1968}, a family of $M\geq N$ quadratic polynomials in $N$ variables realizes at most
  \begin{equation}\label{eq:warren}
    \left(\frac{8eM}{N}\right)^{\!N}
  \end{equation}
  strict sign patterns, that is, patterns in which no polynomial vanishes. See \cite[Chapter~6]{MAT2002} or \cite[Sections~7.3 and~7.6]{BPR2006} for this and its relatives.

  \begin{proof}[Proof of the upper estimate in~\eqref{eq:exd-quantitative}]
    Let $G$ have $n$ vertices and $m$ edges, with $R(G)\leq d$, and set $N=dn$. There are $M=\binom{m}{2}$ polynomials $Q_{uv,xy}$, and none of them vanishes identically, since distinct edges give distinct squared-length forms. We may assume $m>N$, since otherwise the desired upper estimate is immediate. Thus $M\geq N$ and \eqref{eq:warren} applies. Each of the $m!$ edge orders must be realized by a distinct strict sign pattern, whence
    \begin{equation}\label{eq:warren-count-bound}
      m!\leq\left(\frac{8e\binom{m}{2}}{N}\right)^{\!N}\leq\left(\frac{4em^{2}}{N}\right)^{\!N}.
    \end{equation}
    Comparing \eqref{eq:warren-count-bound} with $m!\geq(m/e)^{m}$ forces $m=O(N)$. Writing $m=\rho N$ and taking logarithms, Stirling's formula then turns \eqref{eq:warren-count-bound} into the inequality $\Psi_N(\rho)\leq O(\ln N/N)$, where
    \begin{equation*}
      \Psi_N(\rho)=(\rho-1)\ln N+(\rho-2)\ln\rho-\rho-1-\ln4 .
    \end{equation*}
    For every fixed $\rho>1$, the term $(\rho-1)\ln N$ makes $\Psi_N(\rho)$ positive for all sufficiently large $N$, so the threshold has $\rho=1+o(1)$. Write $\rho=1+\delta$ and $L=\ln N$. Expanding at $\delta=0$ gives
    \begin{equation*}
      \Psi_N(1+\delta)=\delta L-(2+\ln4)-2\delta+\tfrac32\delta^{2}+O(\delta^{3}).
    \end{equation*}
    Hence the zero of $\Psi_N(1+\delta)$ satisfies
    \begin{equation*}
      \delta=\frac{2+\ln4}{L}+O\!\left(\frac1{L^{2}}\right).
    \end{equation*}
    It follows that
    \begin{equation*}
      m\leq N+(2+\ln4)\frac{N}{\ln N}+O\!\left(\frac{N}{(\ln N)^2}\right),
    \end{equation*}
    which is the stated upper estimate.
  \end{proof}

  Nothing in the upper-bound argument depends on $d$ except through $N=dn$, so the upper estimate in~\eqref{eq:exd-quantitative} holds uniformly in $d$. We have thus provided both estimates required to show $\ex_{\R^d}(n)=dn+o(n)$, proving Theorem~\ref{thm:exd}.

  \begin{remark}\label{rmk:algorithmic}
    The polynomials $Q_{uv,xy}$ can also be handled algorithmically, and not only counted. There are exact algorithms which, given a finite family of polynomials, compute a representative point in each of its non-empty sign conditions \cite[Chapter~13]{BPR2006}, \cite{SAFEY2007}. Sorting the $m$ edge lengths at each representative reads off the corresponding order, so for every fixed $d$ one can decide exactly which of the $m!$ orders a given graph realizes in $\R^{d}$, and in particular compute $R(G)$. The cost is enormous and we make no attempt to estimate it. As we will see below, on the line the situation is far better, since there the $Q_{uv,xy}$ factor into linear forms and their sign conditions are the regions of a hyperplane arrangement, hence convex cones.
  \end{remark}

  \begin{remark}\label{rmk:arboricity}
    Theorem~\ref{thm:exd} also gives a lower bound on $R(G)$ for every graph: if $H\subseteq G$ has $n_H$ vertices and $m_H$ edges, then $R(G)\geq R(H)\geq(1-o(1))\,m_H/n_H$ as $n_H\to\infty$, because the upper estimate in~\eqref{eq:exd-quantitative} is uniform in $d$. In the other direction, Avdiukhin, Chatziafratis, Fischer and Yaroslavtsev \cite[Theorem~9]{ACFY2024} show, in the context of contrastive learning, that any consistent set of comparisons between distances whose pairs form a graph of arboricity $r$ can be realized in $\R^{4r}$. Their construction fixes all squared norms and sets the inner products along the edges according to their ranks, so it applies to comparisons between edges with no common vertex. A sufficiently small perturbation preserves the strict inequalities and makes the placement injective. Thus $R(G)\leq4\,a(G)$, where $a(G)=\max_H\lceil m_H/(n_H-1)\rceil$, over subgraphs $H$ with $n_H\geq2$, is the arboricity (Nash-Williams). Hence $R(G)$ and $a(G)$ agree up to a factor of at most $4+o(1)$ as $a(G)\to\infty$.
  \end{remark}

  \section{The extremal problem on the real line}\label{sec:line}

  Theorem~\ref{thm:exd} already gives $\ex_{\R}(n)=n+o(n)$, but in the real line we can say more. First, we settle the case of two disjoint cycles, and then return to the extremal problem.

  \subsection{Two disjoint cycles}\label{sec:cyclesR}

  This subsection is devoted to the proof of Theorem~\ref{thm:twocycles}, which reduces to a single case. The case $(m,n)=(3,3)$ is Proposition~\ref{prop:two-triangles}, so suppose $(m,n)\neq(3,3)$ and fix a total order on the edges of $C_m\sqcup C_n$. If both $m$ and $n$ are larger than $3$, the endpoints of the $\prec$-smallest edge have no common neighbour, so Observation~\ref{obs:easy} lets us contract it, shortening one of the two cycles, and induction on $m+n$ applies. The remaining case is a triangle together with a cycle of length at least $4$, which we treat after two lemmas. Realizing a single cycle on the line is a matter of closing up a walk: give each edge a direction and a length, and the walk returns to where it started exactly when the signed lengths cancel.

  \begin{lemma}\label{lem:signed-cycle}
    Let $C$ be a cycle whose edges, listed in cyclic order, are $e_1,\ldots,e_k$. If positive numbers $\ell_1,\ldots,\ell_k$ and signs $\varepsilon_i\in\{-1,+1\}$ satisfy
    \begin{equation*}
      \sum_{i=1}^{k}\varepsilon_i\ell_i=0,
    \end{equation*}
    then, for every $\tau > 0$, there is a placement of the vertices of $C$ on the line in which the length of $e_i$ lies in $(\ell_i-\tau,\ell_i+\tau)$ for every $i$.
  \end{lemma}

  \begin{proof}
    Set $p_0=0$ and $p_i=p_{i-1}+\varepsilon_i\ell_i$ for $i=1,\ldots,k$. The zero-sum condition says that $p_k=p_0$, so the walk closes up into a cycle, and $\lvert p_i-p_{i-1}\rvert=\ell_i$ for every $i$.

    The placement produced may have coincident vertices, but an arbitrarily small generic perturbation of the vertices makes it injective while changing each length by less than $\tau$.
  \end{proof}

  In every application below we first fix the prescribed lengths of all the components at once, and these are pairwise distinct. We then tacitly take $\tau$ smaller than the least gap between two of them, so that the resulting placements induce exactly the order prescribed by the $\ell_i$.

  We will also need that, for a cycle of length at least $4$, this can be done with all lengths in a short interval.

  \begin{lemma}\label{lem:cycle-short-interval}
    For every $k\geq4$ and every total order
    \begin{equation*}
      e_1\prec e_2\prec\cdots\prec e_k
    \end{equation*}
    on the edges of $C_k$, there are lengths
    \begin{equation*}
      1<\ell_1<\ell_2<\cdots<\ell_k<2
    \end{equation*}
    and signs $\varepsilon_i\in\{-1,+1\}$ such that $\sum_{i=1}^{k}\varepsilon_i\ell_i=0$. Consequently, $C_k$ has a placement on the line realizing $e_1\prec\cdots\prec e_k$ with all edge lengths in $(1,2)$.
  \end{lemma}

  \begin{proof}
    Since the zero-sum condition does not depend on the order in which the edges are listed, Lemma~\ref{lem:signed-cycle} (applied to the edges in cyclic order) produces a placement in which the length of $e_i$ is arbitrarily close to $\ell_i$. As $\ell_1<\cdots<\ell_k$, a close enough placement realizes $e_1\prec\cdots\prec e_k$. It remains to construct the lengths and signs.

    Suppose first that $k=2q$ is even, with $q\geq2$. Give the first four lengths the signs $+,-,-,+$ and, for $3\leq r\leq q$, give the pair $(\ell_{2r-1},\ell_{2r})$ the signs $+,-$. There are as many $+$ signs as $-$ signs, so the signed sum does not change if we translate all the lengths by a common amount, and it depends only on the gaps $\delta_i=\ell_{i+1}-\ell_i$. Indeed, grouping the terms as above,
    \begin{equation*}
      \sum_{i=1}^{k}\varepsilon_i\ell_i
      =
      (\ell_1-\ell_2)+(\ell_4-\ell_3)+\sum_{r=3}^{q}(\ell_{2r-1}-\ell_{2r})
      =
      -\delta_1+\delta_3-\sum_{r=3}^{q}\delta_{2r-1},
    \end{equation*}
    so the zero-sum equation becomes
    \begin{equation*}
      \delta_3=\delta_1+\sum_{r=3}^{q}\delta_{2r-1}.
    \end{equation*}
    Choose every gap other than $\delta_3$ to be an arbitrary positive number, and let the displayed equation determine $\delta_3$, which is then positive as well. Finally set $\ell_1=1+\lambda$ and $\ell_{i+1}=\ell_i+\lambda\delta_i$. The equation is homogeneous in the gaps, so it still holds, and for $\lambda$ small enough every length lies in $(1,3/2)\subset(1,2)$.

    Now suppose that $k=2q+1$ is odd, so $q\geq2$. On the set of increasing tuples $1<\ell_1<\cdots<\ell_k<2$, which is open and connected, define
    \begin{equation*}
      F(\ell_1,\ldots,\ell_k)=\sum_{i=q+2}^{2q+1}\ell_i-\sum_{i=1}^{q+1}\ell_i.
    \end{equation*}
    If all lengths are close to $1$, then $F<0$. If the first $q+1$ lengths are close to $1$ and the last $q$ are close to $2$, then $F$ is close to $2q-(q+1)=q-1>0$. Being continuous on a connected set and taking both signs, $F$ vanishes somewhere. Giving the smallest $q+1$ lengths the sign $-1$ and the largest $q$ lengths the sign $+1$ makes the signed sum equal to $F=0$.
  \end{proof}

  We are now in a position to settle the remaining case, and with it the theorem.

  \begin{proof}[Proof of Theorem~\ref{thm:twocycles}]
    By the reduction above we may assume that the graph is $C_3\sqcup C_n$ with $n\geq4$. Fix an arbitrary total order on its edges, and let $T$ and $D$ denote the triangle and the $n$-cycle, respectively. We distinguish two cases according to the component containing the globally largest edge.

    Suppose first that the globally largest edge belongs to $T$. Apply Lemma~\ref{lem:cycle-short-interval} to the induced order on $E(D)$, obtaining cycle lengths in $(1,2)$. The two nonmaximal edges of $T$ can be assigned distinct lengths $a$ and $b$ in the same interval at their prescribed positions relative to the cycle lengths, since the gaps these finitely many lengths leave in $(1,2)$ are open intervals. Give the largest edge of $T$ the length
    \begin{equation*}
      c=a+b.
    \end{equation*}
    Since $a,b>1$, we have $c>2$, so this edge is indeed globally largest. The equation $c-a-b=0$ realizes the triangle, while Lemma~\ref{lem:cycle-short-interval} realizes $D$.

    Suppose now that the globally largest edge belongs to $D$. List the three triangle edges in their induced order, and assign them lengths
    \begin{equation*}
      1<a<b<\frac32,
      \qquad
      c=a+b<3.
    \end{equation*}
    Assign every other edge except the global maximum a distinct length in $(1,3)$ at its prescribed position. This is possible by filling the open intervals determined by $1,a,b,c,3$. Finally, give the globally largest edge of $D$ the length
    \begin{equation*}
      L = \sum_{e\in E(D)\setminus\{e_{\max}\}}\ell(e).
    \end{equation*}
    There are $n-1\geq3$ summands, all larger than $1$, and hence $L>3$. Thus $L$ is larger than every previously assigned length. The equation $L-\sum_{e\neq e_{\max}}\ell(e)=0$ realizes $D$, while $c-a-b=0$ realizes $T$. Applying Lemma~\ref{lem:signed-cycle} to both components completes the proof.
  \end{proof}

  \subsection{The maximum number of edges of a graph realizable on the line}\label{sec:R1}

  Now we focus on proving Theorem~\ref{thm:ex1}. The general argument of Section~\ref{sec:higher} gives the right first-order term, but the snake of Example~\ref{ex:snake} only yields $n-1$ edges when $d=1$, and the general upper estimate has the constant $2+2\ln2$ in the second-order term. On the line we improve both sides.

  \begin{proof}[Proof of the upper bound in Theorem~\ref{thm:ex1}]
    Let $G$ be a graph on $n$ vertices and $m$ edges with $R(G)=1$. Write a placement of its vertices on the line as
    \begin{equation*}
      x=(x_1,\ldots,x_n)\in\R^n.
    \end{equation*}
    For each edge fix an order of its endpoints, so that for an edge $ij$ the difference $x_i-x_j$ is well-defined. This choice only flips the signs of the two linear factors below and does not affect the argument. For two distinct edges $ij$ and $k\ell$, the comparison
    \begin{equation*}
      |x_i-x_j|<|x_k-x_\ell|
    \end{equation*}
    is equivalent to
    \begin{equation*}
      (x_i-x_j)^2-(x_k-x_\ell)^2<0,
    \end{equation*}
    and the expression on the left factors as
    \begin{equation*}
      (x_i-x_j-x_k+x_\ell)(x_i-x_j+x_k-x_\ell).
    \end{equation*}
    Thus the comparison between $ij$ and $k\ell$ is determined by the signs of two linear forms. Ranging over all unordered pairs of edges produces at most
    \begin{equation*}
      h=2\binom{m}{2}=m(m-1)
    \end{equation*}
    linear hyperplanes. Distinct pairs of edges may determine the same hyperplane, which only helps, since the region bound below increases with the number of hyperplanes.

    All these hyperplanes contain the diagonal subspace
    \begin{equation*}
      \Delta=\{x_1=\cdots=x_n\},
    \end{equation*}
    because translating every point by the same amount does not change any edge length. Taking the quotient by $\Delta$, we obtain a central arrangement of at most $h$ hyperplanes in a vector space of dimension $n-1$, which by the standard region bound of Zaslavsky for central arrangements \cite{ZASLAVSKY1975} has at most $2\sum_{r=0}^{n-2}\binom{h-1}{r}$ regions. Each region determines the signs of all the linear factors above and therefore at most one total order of the edge lengths. Antipodal regions correspond to the placements $x$ and $-x$, which induce exactly the same edge lengths, so they may be counted once. The number of edge orders realizable by $G$ is therefore at most
    \begin{equation*}
      \sum_{r=0}^{n-2}\binom{h-1}{r}.
    \end{equation*}
    Since every one of the $m!$ total orders on $E(G)$ must be realizable, we obtain
    \begin{equation}\label{eq:arrangement-count-bound}
      m!
      \leq
      \sum_{r=0}^{n-2}\binom{m(m-1)-1}{r}.
    \end{equation}

    We estimate the largest $m$ for which \eqref{eq:arrangement-count-bound} can hold. We may assume $m\geq n$, as the bound is trivial otherwise. Bounding the sum by $n-1$ times its largest term and using $\binom{m(m-1)-1}{r}\leq\binom{m^2}{n-2}$ for $r\leq n-2$ gives $m!\leq (n-1)\binom{m^2}{n-2}\leq n\,m^{2n}$, which together with $m!\geq(m/e)^m$ forces $m=O(n)$. Note also that $h-1\geq n(n-1)-1$.

    In the relevant range the last term of the sum dominates, since
    \begin{equation*}
      \frac{\binom{h-1}{r-1}}{\binom{h-1}{r}}
      =
      \frac{r}{h-r}
      =O\!\left(\frac1n\right)
    \end{equation*}
    uniformly for $r\leq n-2$. Hence
    \begin{equation*}
      \sum_{r=0}^{n-2}\binom{h-1}{r}
      =
      \binom{h-1}{n-2}\left(1+O\!\left(\frac1n\right)\right).
    \end{equation*}
    Writing $m=\rho n$ and using Stirling's formula gives
    \begin{align*}
      \ln\binom{h-1}{n-2}
      &=n\ln n+2n\ln\rho+n+O(\ln n),\\
      \ln(m!)
      &=\rho n\ln n+\rho n\ln\rho-\rho n+O(\ln n).
    \end{align*}
    Consequently,
    \begin{equation*}
      \ln(m!)-\ln\binom{h-1}{n-2}
      =
      n\Phi_n(\rho)+O(\ln n),
    \end{equation*}
    where
    \begin{equation*}
      \Phi_n(\rho)
      =
      (\rho-1)\ln n+(\rho-2)\ln\rho-\rho-1.
    \end{equation*}
    For every fixed $\rho>1$, this expression is positive for all sufficiently large $n$, so the threshold has $\rho=1+o(1)$. Writing $\rho=1+\delta$ and $L=\ln n$ and expanding at $\delta=0$ yields
    \begin{equation*}
      \Phi_n(1+\delta)
      =
      \delta L-2-2\delta+\frac32\delta^2+O(\delta^3),
    \end{equation*}
    whose zero satisfies $\delta=\frac{2}{L}+O\!\left(\frac1{L^2}\right)$. It follows from \eqref{eq:arrangement-count-bound} that
    \begin{equation*}
      m
      \leq
      n+\frac{2n}{\ln n}
      +O\!\left(\frac{n}{(\ln n)^2}\right),
    \end{equation*}
    as claimed.
  \end{proof}

  For the lower bound we take $k$ long disjoint cycles and join them by bridges. A cycle has as many edges as vertices, and each bridge adds one more edge, so the gain over $n$ is essentially the number of cycles, and how many cycles fit into $n$ vertices depends on how long they must be for their disjoint union to be realizable. The following elementary balancing lemma, a one-dimensional discrepancy estimate, is what controls this. See, for example, \cite[Chapter~1]{MATOUSEK1999} for general background on discrepancy.

  \begin{lemma}\label{lem:one-dimensional-balancing}
    Let $Y>0$ and let $y_1,\ldots,y_q\in[1,Y]$ with $q>Y$. Then there are signs $\varepsilon_i\in\{-1,+1\}$ such that
    \begin{equation*}
      -2Y
      \leq
      Y+\sum_{i=1}^{q}\varepsilon_i y_i
      <0.
    \end{equation*}
  \end{lemma}

  \begin{proof}
    Let $W=y_1+\cdots+y_q$. Since $q>Y$ and every $y_i\geq1$, we have $W>Y$, and in particular $W>(Y+W)/2$. Starting with the empty set, add indices to $A$ until the partial sum $U=\sum_{i\in A}y_i$ first exceeds $(Y+W)/2$. The last summand added is at most $Y$, and therefore
    \begin{equation*}
      \frac{Y+W}{2}<U\leq\frac{Y+W}{2}+Y.
    \end{equation*}
    Set $\varepsilon_i=-1$ for $i\in A$ and $\varepsilon_i=+1$ otherwise. Then $Y+\sum_i\varepsilon_i y_i=Y+W-2U$, which lies in $[-2Y,0)$.
  \end{proof}

  We can now show that a disjoint union of $k$ cycles, each of them long enough in terms of $k$, is realizable. Recall that Theorem~\ref{thm:manycycles} states that, for
  \begin{equation*}
    K(k)=2\cdot3^{k-1}+2,
  \end{equation*}
  the disjoint union $kC_{K(k)}$ is realizable in $\R$.

  The construction is as follows. Assign to the edges nearly equal lengths, contained in a small window above $1$ and spaced by a common step $\eta$ in the prescribed order. By Lemma~\ref{lem:signed-cycle}, closing up a cycle amounts to choosing signs on its edges whose signed sum is zero, and with many nearly equal lengths this can almost be done by taking half of the signs positive and half negative. Lemma~\ref{lem:one-dimensional-balancing} bounds the error that remains, and we correct it by increasing the length of the longest edge of the cycle by a suitable gap. Such an increase shifts every length above it, so we treat the cycles in the order of their longest edges, and no correction then affects a cycle already closed. Each gap is bounded in terms of the lengths assigned so far, and these grow by a factor of three per cycle, which is why the cycles must be exponentially long in their number.


  \begin{proof}[Proof of Theorem~\ref{thm:manycycles}]
    Fix an arbitrary total order on the edges of $kC_{K(k)}$. For brevity write $K=K(k)$ and $m=kK$, and denote the edges, in the prescribed order, by $e_1\prec e_2\prec\cdots\prec e_m$. Call the longest edge of a cycle its \emph{top edge}, and let the cycles be $D_1,\ldots,D_k$, relabelled so that the top edge of $D_1$ comes first in the order, then that of $D_2$, and so on. Write $t_j$ for the top edge of $D_j$.

    We assign the lengths from left to right. Choose $\eta>0$ with $\eta m<1$, give $e_1$ the length $\ell_1=1$, and let each successive length exceed the previous one by $\eta$. The only exception is the step into a top edge $t_j$, which is instead $\eta+g_j$ for a positive \emph{special gap} $g_j$, chosen so that $D_j$ admits a signing whose signed sum is zero. The gaps $g_1,\ldots,g_k$ are chosen in this order.

    Suppose $g_1,\ldots,g_{j-1}$ have been chosen. Every edge of $D_j$ other than $t_j$ has then received its final length, and we give $t_j$ the provisional length $Y_j$ that the step $\eta$ alone would produce. All steps up to $t_j$ equal $\eta$ apart from the earlier special gaps, and $\eta m<1$, so
    \begin{equation*}
      Y_j<2+g_1+\cdots+g_{j-1},
    \end{equation*}
    while every other edge of $D_j$ has length in $[1,Y_j]$.

    We show by induction on $j$ that
    \begin{equation*}
      Y_j<2\cdot3^{j-1}
      \qquad\text{and}\qquad
      g_j\leq 2Y_j .
    \end{equation*}
    The bound on $Y_j$ follows from the bounds already established for smaller indices, since
    \begin{equation*}
      Y_j<2+\sum_{i<j}g_i\leq2+\sum_{i<j}2Y_i<2+\sum_{i<j}4\cdot3^{i-1}=2\cdot3^{j-1},
    \end{equation*}
    the sum being empty when $j=1$. In particular $Y_j<2\cdot3^{k-1}<K-1$, so Lemma~\ref{lem:one-dimensional-balancing} applies to the $K-1$ lengths of $D_j$ other than $t_j$, all of which lie in $[1,Y_j]$, taking $Y=Y_j$ and $q=K-1$. It returns signs for those edges whose signed sum $\sigma_j$ satisfies
    \begin{equation*}
      -2Y_j\leq Y_j+\sigma_j<0 .
    \end{equation*}
    Setting $g_j=-(Y_j+\sigma_j)$ therefore gives a gap that is positive and at most $2Y_j$, completing the induction. Giving $t_j$ the sign $+1$ and its final length $Y_j+g_j$, the signed sum on $D_j$ becomes $(Y_j+g_j)+\sigma_j=0$, as required.

    Continuing from left to right produces increasing lengths $\ell_1<\ell_2<\cdots<\ell_m$ and, on every cycle, a signing whose signed sum is zero. Once the equation for $D_j$ has been obtained it is never disturbed, since every later special gap falls above $t_j$, whereas every edge of $D_j$ lies at or below it.

    Lemma~\ref{lem:signed-cycle} now realizes every cycle with the assigned lengths. Translating the components independently if necessary makes the placement injective, and the prescribed global order is realized.
  \end{proof}

  \begin{proof}[Proof of the lower bound in Theorem~\ref{thm:ex1}]
    Choose the largest integer $k$ such that
    \begin{equation*}
      kK(k)\leq n.
    \end{equation*}
    By Theorem~\ref{thm:manycycles}, the graph $kC_{K(k)}$ is realizable on the line. Connect its $k$ components by $k-1$ bridges. If $n-kK(k)$ vertices remain, attach them one at a time as leaves. By the bridge observation in Section~\ref{sec:basic}, the resulting graph is still realizable on the line. It has $n$ vertices and
    \begin{equation*}
      kK(k)+(k-1)+(n-kK(k))=n+k-1
    \end{equation*}
    edges.

    By maximality of $k$,
    \begin{equation*}
      n<(k+1)K(k+1)=\Theta(k3^k).
    \end{equation*}
    It follows that
    \begin{equation*}
      k
      \geq
      \frac{\ln n}{\ln3}-O(\ln\ln n).
    \end{equation*}
    Consequently,
    \begin{equation*}
      \ex_{\R}(n)
      \geq n+k-1
      \geq
      n+\frac{\ln n}{\ln3}-O(\ln\ln n).
    \end{equation*}
  \end{proof}

  \section{Final remarks and open problems}\label{sec:final}

  The results of this paper leave open two directions. The first is extremal, and concerns the error terms $E_d(n)$, where our bounds are still far apart: on the line they leave everything between $\ln n/\ln3$ and $2n/\ln n$.

  \begin{problem}\label{prob:error-term}
    Determine the order of magnitude of the error term
    \[
    E_d(n)=\ex_{\R^d}(n)-dn
    \]
    for fixed $d$. In particular, determine the order of growth of $E_1(n)$.
  \end{problem}

  The second direction is structural: which graphs are realizable, and in which dimension? Section~\ref{sec:basic} shows that forests and cycles are realizable on the line, Proposition~\ref{prop:two-triangles} shows that a realizable graph contains at most one triangle, and Theorem~\ref{thm:ex1} shows that such graphs are asymptotically sparse. These conditions are far from sufficient, and the results of Section~\ref{sec:line} suggest that what governs realizability on the line is the interaction between distinct cycles rather than any single one of them.

  \begin{problem}\label{prob:characterize-line}
    Characterize the finite graphs $G$ that are realizable in $\R$.
  \end{problem}

  For complete graphs the problem goes back to \cite{BL05}, and the question below is the one posed in \cite{AM22}. We determined $R(K_n\sqcup K_n)=n-1$, but $R(K_n)$ itself remains open.

  \begin{problem}\label{prob:complete-graphs} Is $R(K_n)=n-2$ for every $n\geq7$? \end{problem}

  The answer is yes for $n\leq6$: the case $n=4$ is in \cite{AM22}, and for $n=5,6$ the lower bound follows from Maehara's theorem, as explained in the introduction. In general the best bounds are $\lfloor n/2\rfloor+1\leq R(K_n)\leq n-2$.

  Finally, for complete bipartite graphs, \cite{AM22} and \cite{MRR26} settle every case except the balanced one.

  \begin{problem}\label{prob:balanced-bipartite}
    For $n\geq 4$, determine $R(K_{n,n})$.
  \end{problem}

  \section*{Acknowledgments}

  The authors, particularly Gerardo and Miguel, would like to thank the organizers of ``Escuela Queretana de Matemáticas 2026''. We also thank Crisanto Salazar, who took part in this project during that school and contributed the idea of contracting the $\prec$-smallest edge (Observation~\ref{obs:easy}).

  This work was supported by UNAM-PAPIIT project IN119026 and by UNAM-PAPIIT project IN114726. The first author received financial support by the UNAM-DGAPA Postdoctoral Fellowship Program. The second author received support from UNAM DGAPA-PASPA to work on this topic during a sabbatical stay at the Instituto de Matemáticas, Unidad Juriquilla, UNAM.

  The authors acknowledge the use of Claude (Anthropic) and ChatGPT (OpenAI) to assist with exploration, drafting and reorganization of parts of the manuscript, including the figures, to improve the clarity and presentation of the exposition, and to help implement simple algorithms used during the exploratory stages of the work. All mathematical content, computational output relevant to the final results, and final formulations were reviewed and verified by the authors, who approve the resulting manuscript and take full responsibility for its contents.

  \bibliographystyle{elsarticle-num}
  \bibliography{refs}

\end{document}

%% file: figuras/snake.tex
\begin{tikzpicture}[
    oldE/.style={black!70,thick},
    newE/.style={intc,very thick,line cap=round},
    face/.style={fill=black!8},
    seedface/.style={fill=hc!35},
    newface/.style={fill=intc!12},
    plabel/.style={font=\footnotesize},
  ]

  \begin{scope}[shift={(0,-0.693)}]
    \coordinate (p0) at (0.00,0.000);
    \coordinate (p1) at (0.80,1.386);
    \coordinate (p2) at (1.60,0.000);
    \coordinate (p3) at (2.40,1.386);
    \coordinate (p4) at (3.20,0.000);
    \coordinate (p5) at (4.00,1.386);
    \coordinate (p6) at (4.80,0.000);

    \fill[seedface] (p0)--(p1)--(p2)--cycle;
    \fill[face]     (p1)--(p2)--(p3)--cycle;
    \fill[face]     (p2)--(p3)--(p4)--cycle;
    \fill[face]     (p3)--(p4)--(p5)--cycle;
    \fill[newface]  (p4)--(p5)--(p6)--cycle;

    \draw[oldE] (p0)--(p1) (p0)--(p2) (p1)--(p2)
                (p1)--(p3) (p2)--(p3)
                (p2)--(p4) (p3)--(p4)
                (p3)--(p5) (p4)--(p5);
    \draw[newE] (p4)--(p6) (p5)--(p6);

    \foreach \i in {0,...,5}{\node[vtx] at (p\i) {};}
    \node[vtx,fill=intc] at (p6) {};
    \node[intc,plabel,right=2pt] at (p6) {$v$};
  \end{scope}
  \node[plabel,align=center] at (2.40,-1.72)
    {$d=2$: a strip of triangles\\[-1pt] $2n-3$ edges};

  %
  \begin{scope}[shift={(8.0,0.174)}]
    \coordinate (q0) at (-1.584, 0.902);
    \coordinate (q1) at (-1.406,-1.263);
    \coordinate (q2) at (-0.509, 0.915);
    \coordinate (q3) at ( 0.691,-0.194);
    \coordinate (q4) at ( 0.768,-1.263);
    \coordinate (q5) at ( 2.039, 0.902);

    \filldraw[fill=intc!13,draw=black!70,thick,line join=round] (q2)--(q4)--(q5)--cycle;
    \draw[newE] (q2)--(q5) (q4)--(q5);
    \filldraw[fill=black!14,draw=black!70,thick,line join=round] (q1)--(q2)--(q4)--cycle;
    \filldraw[fill=intc!17,draw=black!70,thick,line join=round] (q3)--(q4)--(q5)--cycle;
    \draw[newE] (q3)--(q5) (q4)--(q5);
    \filldraw[fill=intc!17,draw=black!70,thick,line join=round] (q2)--(q3)--(q5)--cycle;
    \draw[newE] (q2)--(q5) (q3)--(q5);
    \filldraw[fill=black!14,draw=black!70,thick,line join=round] (q1)--(q3)--(q4)--cycle;
    \filldraw[fill=hc!38,draw=black!70,thick,line join=round] (q0)--(q1)--(q2)--cycle;
    \filldraw[fill=hc!54,draw=black!70,thick,line join=round] (q0)--(q2)--(q3)--cycle;
    \filldraw[fill=hc!38,draw=black!70,thick,line join=round] (q0)--(q1)--(q3)--cycle;

    \foreach \i in {0,...,4}{\node[vtx] at (q\i) {};}
    \node[vtx,fill=intc] at (q5) {};
    \node[intc,plabel,right=2pt] at (q5) {$v$};
  \end{scope}
  \node[plabel,align=center] at (8.0,-1.72)
    {$d=3$: a snake of tetrahedra\\[-1pt] $3n-6$ edges};

\end{tikzpicture}

%% file: figuras/triangles-disjoint.tex
\begin{tikzpicture}[scale=0.9,lbl/.style={fill=white,inner sep=1.5pt}]
  \coordinate (Yl) at (0,0);
  \coordinate (Yr) at (2.4,0);
  \coordinate (Yt) at (1.2,1.8);
  \draw[thick] (Yl) -- node[lbl] {$6$} (Yr);
  \draw[thick] (Yl) -- node[lbl] {$1$} (Yt);
  \draw[thick] (Yr) -- node[lbl] {$2$} (Yt);
  \foreach \p in {Yl,Yr,Yt}{\node[vtx] at (\p) {};}
  \begin{scope}[shift={(2.8,0)}]
    \coordinate (Xl) at (0,0);
    \coordinate (Xr) at (2.4,0);
    \coordinate (Xt) at (1.2,1.8);
    \draw[thick] (Xl) -- node[lbl] {$5$} (Xr);
    \draw[thick] (Xl) -- node[lbl] {$3$} (Xt);
    \draw[thick] (Xr) -- node[lbl] {$4$} (Xt);
    \foreach \p in {Xl,Xr,Xt}{\node[vtx] at (\p) {};}
  \end{scope}
\end{tikzpicture}

%% file: figuras/triangles-shared.tex
\begin{tikzpicture}[scale=0.9,lbl/.style={fill=white,inner sep=1.5pt}]
  \coordinate (T) at (0,1.4);
  \coordinate (B) at (0,-1.4);
  \coordinate (L) at (-2,0);
  \coordinate (R) at (2,0);
  \draw[thick] (T) -- node[lbl] {$5$} (B);   
  \draw[thick] (L) -- node[lbl] {$1$} (T);
  \draw[thick] (L) -- node[lbl] {$2$} (B);
  \draw[thick] (R) -- node[lbl] {$3$} (T);
  \draw[thick] (R) -- node[lbl] {$4$} (B);
  \foreach \p in {T,B,L,R}{\node[vtx] at (\p) {};}
\end{tikzpicture}

%% file: figuras/order.tex
\begin{tikzpicture}
  \coordinate (a1) at (0,0.7);
  \coordinate (a2) at (0,-0.7);
  \coordinate (b1) at (2.6,1.0);
  \coordinate (b2) at (3.0,0);
  \coordinate (b3) at (2.6,-1.0);
  \foreach \x in {a1,a2}{\foreach \y in {b1,b2,b3}{\draw[crossE] (\x)--(\y);}}
  \draw[intE] (a1)--(a2);
  \draw[intE] (b1)--(b2); \draw[intE] (b2)--(b3); \draw[intE] (b1)--(b3);
  \draw[black!60] (0,0) ellipse (0.5 and 1.3);
  \draw[black!60] (2.75,0) ellipse (0.85 and 1.55);
  \node at (0,1.6) {$A$};
  \node at (2.75,1.9) {$B$};
  \foreach \p in {a1,a2,b1,b2,b3}{\node[vtx] at (\p) {};}
  \node at (1.4,-1.95) {$G$};
  \begin{scope}[shift={(6.3,0)}]
    \foreach \i in {0,...,4}{\coordinate (h\i) at ({90+72*\i}:1.25);}
    \foreach \i in {0,...,4}{\foreach \j in {0,...,4}{\draw[hE] (h\i)--(h\j);}}
    \foreach \i in {0,...,4}{\node[vtx] at (h\i) {};}
    \node at (0,-1.95) {$H$};
  \end{scope}
  \begin{scope}[shift={(8.6,0)}]
    \draw[->] (0,0) -- (5.5,0) node[below left=1pt,font=\footnotesize] {length};
    \draw[crossc,line width=5pt] (0.2,0) -- (1.6,0);
    \draw[hc,line width=5pt] (2.0,0) -- (3.4,0);
    \draw[intc,line width=5pt] (3.8,0) -- (5.2,0);
    \node[crossc,above,align=center,font=\footnotesize] at (0.9,0.15) {cross\\ edges of $G$};
    \node[hc,above,align=center,font=\footnotesize] at (2.7,0.15) {all edges\\ of $H$};
    \node[intc,above,align=center,font=\footnotesize] at (4.5,0.15) {internal\\ edges of $G$};
    \node[crossc,below,font=\footnotesize] at (1.6,-0.12) {$s$};
    \node[hc,below,font=\footnotesize] at (2.0,-0.12) {$S$};
    \node[hc,below,font=\footnotesize] at (3.4,-0.12) {$\ell$};
    \node[intc,below,font=\footnotesize] at (3.8,-0.12) {$\cL$};
  \end{scope}
\end{tikzpicture}